\documentclass[10pt,a4paper]{article}
\usepackage[T1]{fontenc}
\usepackage{amsthm}
\usepackage{amssymb}
\usepackage{amsmath}
\usepackage{mathabx}
\usepackage[notcite, notref]{showkeys}
\usepackage{xcolor}
\usepackage{verbatim, color}
\usepackage{bbm}
\usepackage{float}
\usepackage{dsfont, graphicx}
\usepackage{epstopdf}
\usepackage[hidelinks]{hyperref}
\usepackage[T1]{fontenc}
\usepackage{amssymb,amsmath, amsthm, amsfonts}
\usepackage{graphicx}
\usepackage[hidelinks]{hyperref}

\usepackage{thmtools,thm-restate}

\makeatletter
\numberwithin{equation}{section}
\numberwithin{figure}{section}
\theoremstyle{plain}
\newtheorem{thm}{Theorem}[section]

\newtheorem{lem}[thm]{Lemma}

\newtheorem{rem}[thm]{Remark}

  \newcounter{casectr}
  
\theoremstyle{definition}

\theoremstyle{remark}

\newcommand{\NNN}{\mathcal{N}}

\newcommand{\<}{\langle}
\renewcommand{\>}{\rangle}

\begin{document}
\title{On a large deviation principle for 1d cubic NLS with optimal decaying data}
\author{Chenjie Fan\footnote{State Key Laboratory of Mathematical Sciences, Academy of Mathematics and Systems Science, Chinese Academy of Sciences,
Beijing, China,
fancj@amss.ac.cn}, Feng Ye\footnote{School of Mathematical Sciences,
University of Chinese Academy of Sciences,
Beijing, China
yefeng22@mails.ucas.ac.cn}}
\maketitle

\begin{abstract}
In this article, we revisit the work of \cite{garrido2023large}, and prove large deviation principles for more general random initial data for cubic NLS. The Fourier coefficient of our random data admits an optimal polynomial decay.
\end{abstract}

\section{Introduction}
\subsection{Statement of results}
Consider cubic NLS on 1d torus $\mathbb{T}$,
\begin{equation}\label{eq: main}
\begin{cases}
iu_{t}+\Delta u=\epsilon^{2}|u|^{2}u,\\	
u(0,x)=u_{0}(x)
\end{cases}
\end{equation}
here $\epsilon>0$ is a parameter.

We will focus on the random initial data of type
\begin{equation}\label{eq: rmain}
u^{\omega}_{0}(x)=\sum_{n}c_{n}g_{n}^{\omega}e^{inx}
\end{equation}
where $c_{n}=\<n\>^{-(\frac{1}{2}+\theta)}, \theta>0$, and $g_{n}^{\omega}$ be i.i.d standard complex Gaussian.\\

Note that $u^{\omega}_{0}\in L_{x}^{2}$ almost surely, and it is well known that all $L^{2}$ initial data generates a global flow, \cite{bourgain1993fourier}. Our main result is 
\begin{thm}\label{thm: main}
Consider \eqref{eq: main} with random initial data \eqref{eq: rmain}. Let $c_{n}$ in \eqref{eq: rmain} be $\langle n\rangle^{-(\frac{1}{2}+\theta)},\theta>0$. Let $T_{\epsilon}\approx \epsilon^{-1}$. One has the following large deviation principle
\begin{equation}\label{eq: estimatemain}
\lim_{\epsilon\rightarrow 0^{+}}\epsilon\ln \mathbb{P}(\|u^{\omega}\|_{L^{\infty}_{t,x}([0,T_{\epsilon}]\times \mathbb{T})}>z_{0}\epsilon^{-\frac{1}{2}})=-\frac{z_{0}^{2}}{\sum_{n}|c_{n}|^{2}}.
\end{equation}	
\end{thm}
\begin{rem}
Upon finishing of this work, we notice a recent preprint of Liang-Wang, \cite{liang2025random},  which handles Theorem \ref{thm: main} in the case $\theta>\frac{1}{2}$. It should be noted that when $\theta>\frac{1}{2}$, the Fourier coefficient is in $l^{1}$, which forms an algebra. We handle all $\theta>0$.\\

Both \cite{garrido2023large}, \cite{liang2025random}, and this work actually work for time scale for $c\epsilon^{-1}|\ln \epsilon|$, for some $c$ small but universal. Such improvement comes from standard improvement in Gronwall argument. In \cite{liang2025random}, they state the time scale as $O(\epsilon^{-1}|\ln \epsilon|)$, however, it seems hard to cover time scale for $C\epsilon^{-1}|\ln \epsilon|$ for $C$ large, see the line below (3.30) in \cite{liang2025random}.
\end{rem}

\subsection{Background and motivation}
This work is inspired by the pioneering work \cite{garrido2023large}. Motivated by rogue wave phenomena, Garrido, Grande, Kurianski and Staffilani study \eqref{eq: main}, with random initial data of form \eqref{eq: rmain}, but for $c_{n}=ae^{-b|n|}$. They essentially
\footnote{They state the result for pointwise in $t\leq T_{\epsilon}\sim \epsilon^{-1}$. We take this chance to cover $\|u\|_{L_{t,x}^{\infty}}$ rather than estimate pointwise in $t$.} obtained \eqref{eq: estimatemain} for such types of random initial data. It has been remarked in \cite[Remark 1.3]{garrido2023large}, \textit{finding optimal family of coefficients (in terms of decay)}, rather than considering exponential decay $c_{n}$, remains an interesting open question.

By combining random data type techniques in \cite{bourgain1994periodic,bourgain1996invariant} and some ideas in \cite{compaan2021pointwise} (which also goes back to Bourgain's 1990s seminal work), we are able to extend the result in \cite{garrido2023large} to coefficients with polynomial decay $\langle n\rangle^{-\frac{1}{2}-\theta}$, $\theta>0$. Such constraint is optimal in the sense when $
\theta\leq 0$, even the initial data will leave $L_{x}^{\infty}$. It remains an interesting question to go beyond the time scale $T_{\epsilon}$. We note that this $\epsilon^{-1}$ time scale (neglecting logarithm corrections) is in some sense critical and we refer to \cite[Remark 1.7]{garrido2023large} for more details.
We note that both current work and \cite{garrido2023large},  can cover time scale of form $c\frac{|\ln \epsilon|}{\epsilon}$, where $c$ is some small universal number.\\

The overall proof scheme of current article follows \cite{garrido2023large}, which contains a large deviation principle for the (modified) linear flow and a smoothing estimate for the Duhamel part, (there is some subtle part here since one needs to first slightly perturb the free linear flow, we will discuss it later).\\

 Our main contributions in this work are
\begin{itemize}
\item We find a rather straightforward proof for the large deviation principle for the linear flow, which simplifies the arguments in \cite{garrido2023large}, and more importantly, extends to random initial data of form $\sum_{n}\frac{g^\omega _n}{\<n\>^{\frac{1}{2}+\theta}}e^{inx}$, $\forall \theta>0$. Our techniques for this part are crucial for our further analysis in the nonlinear smoothing part.
\item We adapt the $X^{s,b}$ analysis in \cite{bourgain1994periodic,bourgain1996invariant} in the setting of long time analysis of \cite{garrido2023large}, which gives the desired nonlinear smoothing\footnote{On one hand, we will use the $X^{s,b}$ analysis in the local random data theory as a black box;  on the other hand, the long time analysis we use to get nonlinear smoothing does explicitly use computations of $X^{s,b}$ type, but without going to the explicit form of $X^{s,b}$ analysis.}.
\end{itemize}

We end this session by pointing out ever since the seminal work \cite{bourgain1994periodic,bourgain1996invariant}, random initial data theory in nonlinear dispersive PDEs has become a very active research field. Many researchers contribute to this field and it is impossible to survey the field here. We simply mention that  in the perspective of  regularity in NLS model, the recent breakthrough \cite{deng2022random} gives a quite complete answer. We refer to \cite{deng2022random}, \cite{garrido2023large} and reference therein for more related reference.

\subsection{Notation}
Many terms $g_{n}^{\omega}, u^{\omega}, u_{0}^{\omega}$ are random, and thus depend on $\omega$. For simplicity, we will hide the implicit $\omega$ and simply denote them as $g_{n}, u,u_{0}$.\\

We introduce $\NNN_{1}, \NNN_{2}$ as 
\begin{equation}\label{eq: n1}
\NNN_{1}(f_{1},f_{2},f_{3})=\sum_{n_{2}\neq n_{1}, n_{3}}\hat{f}_{1}(n_{1})\overline{\hat{f}_{2}}(n_{2})\hat{f}_{3}(n_{3})e^{i(n_{1}-n_{2}+n_{3})x},
\end{equation}
\begin{equation}\label{eq: n_{2}}
\NNN_{2}(f_{1},f_{2},f_{3})=\sum_{n}\hat{f}_{1}(n)\overline{\hat{f}_{2}}(n)\hat{f}_{3}(n)e^{inx}	
\end{equation}
and one has 
\begin{equation}\label{eq: algebra}
|f|^{2}f-2\frac{1}{2\pi}(\int |f|^{2})f=\NNN_{1}(f,f,f)-\NNN_{2}(f,f,f)	
\end{equation}
We sometimes short $\NNN_{i}(f,f,f)$ as $\NNN_{i}(f)$.
\section*{Acknowledgement}
C. Fan was partially supported by the National Key R\&D Program of China, 2021YFA1000800, CAS Project for Young Scientists in Basic Research, Grant No.YSBR-031, and NSFC Grant (Nos. 12288201 \& 12471232).
\section{Preliminary}
\subsection{$X^{s,b}$ space}
In this section, we recall that  $X^{s,b}$  norm is  defined by
\begin{equation*}
    \left\|u\right\|_{X^{s,b}}:= \left(\int_{\mathbb{R}} \sum_{n}\<\tau + n^2\>^{2b} \<n\>^{2s} |\hat{u}(\tau,n)|^2 d\tau\right)^{\frac{1}{2}},
\end{equation*}
where $\hat{u}(\tau,n)$ is the spacetime Fourier transformation of $u$. One defines  $X^{s,b}[0,T]$ local in time  by
\begin{equation*}
    \left\|u\right\|_{X^{s,b}[0,T]}:= \inf_{\tilde{u}=u, t\in [0,T]} \left\|\tilde{u}\right\|_{X^{s,b}}.
\end{equation*}

These kinds of spaces were first introduced by Bourgain \cite{bourgain1993fourier} and are of essential use in random data  theory, \cite{bourgain1994periodic}, \cite{bourgain1996invariant}. We will only use them as a black box. We note that 
\begin{itemize}
\item $X^{s,b}[0,T] \subset C_tH^s([0,T]\times \mathbb{T})$ for any $b>\frac{1}{2}$.
\item $X^{s,b}$ space is useful, among others, because of the following $X^{s,b}$ smoothing estimate,
\begin{equation}
	 \left\|\eta(t) \int_{0}^{t} e^{i(t-s)\Delta}F(s)ds\right\|_{X^{s,b+1}[0,1]}
    \lesssim \left\|F\right\|_{X^{s,b}[0,1]}.
\end{equation}
\end{itemize}

\subsection{Hyper-contractivity estimate} 
We recall the following  well-known hypercontractivity of multi-Guassian from \cite{tzvetkov2010construction}, see also \cite{bourgain1996invariant}, \cite{fan20212d}. Those estimates are widely used in random data theory, and is probably well known in the probability community.\\

 We record it here for the convenience of readers.
\begin{lem}\label{lem: hyper}
    Let $ \{g^w_{n}\} $ be i.i.d complex Gaussian, let $ \{c_{n_{1},\ldots,n_{k}}\} $ be (fixed and deterministic) complex numbers. Let 
    \begin{align*}
        F_{k}(w)=\sum \limits_{n_{1},\ldots,n_{k}}c_{n_{1},\ldots,n_{k}}g^w_{n_{1}}\cdots g^w_{n_{k}}.
    \end{align*}
    Then there holds the associated large deviation type estimate
    \begin{align*}
        \mathbb{P}\{|F_{k}| > \lambda \|F_{k}\|_{L^{2}(\Upsilon )}\}\leq e^{-C\lambda ^{{2}/{k}}},\quad\forall\lambda>0.
    \end{align*}
\end{lem}

\section{An overview of the proof}
Let $u, u_{0}$ be as in Theorem \ref{thm: main}.  Recall 
\begin{equation}\label{eq: u0}
u_{0}=\sum_{n}c_{n}g_{n}e^{inx}=\sum_{n}\frac{g_{n}}{\<n\>^{\frac{1}{2}+\theta}}e^{inx}.
\end{equation}

We will fix $\theta>0$. We will only consider $\theta\ll 1$ and the arguments easily extend to the case when $\theta$ takes larger value.\\

 For every $\epsilon>0$, we fix  $T_{\epsilon}\sim \epsilon^{-1}$.\\

The overall strategy to approach Theorem \ref{thm: main},  as in \cite{garrido2023large}, is to prove a precise LDP for the (modified) linear flow, and prove that the difference of $u$ and the (modified) linear flow is small with probability close to $1$.\\

Let modified linear flow $u_{app}$ be defined as
\begin{equation}\label{eq: uapp}
u_{app}(t,x):=e^{-2it\epsilon^{2}\frac{1}{2\pi}\|u_{0}\|_{L_{x}^{2}}^{2}}\sum_{n}c_{n}g_{n}e^{inx}e^{it\epsilon^{2}c_{n}^{2}|g_{n}|^{2}-itn^{2}}.
\end{equation}
(Strictly speaking, $u_{app}$ depends on $\epsilon$, for notation simplicity, we hide this $\epsilon$.)

Note that $u_{app}$ solves
\begin{equation}\label{eq: uappformula}
\begin{cases}
i\partial_{t}u_{app}+\Delta u_{app}=-\epsilon^{2}\NNN_{2}(u_{app})+2\epsilon^{2}\frac{1}{2\pi}\|u_{0}\|_{L_{x}^{2}}	u_{app},\\
u_{app}(0,x)=u_{0}(x)
\end{cases}
\end{equation}
We refer to \eqref{eq: n1}, \eqref{eq: n_{2}}, \eqref{eq: algebra} for the notation $\NNN_{2}$.

Theorem \ref{thm: main} follows from the following two lemmas.

\begin{lem}\label{lem: modlinearldp}
Let $z_{0}>0$, $T_{\epsilon}=O(\epsilon^{-1})$ and $u_{app}$ be defined as \eqref{eq: uapp}, one has the large deviation principle,
\begin{equation}\label{eq: modlinearldp}
\lim_{\epsilon\rightarrow 0^{+}}\epsilon \ln \mathbb{P}(\|u_{app}\|_{L_{t,x}^{\infty}([0,T_{\epsilon}]\times \mathbb{T})}>z_{0}\epsilon^{-\frac{1}{2}})=-\frac{z_{0}^{2}}{\sum_{n}c_{n}^{2}}	
\end{equation}

\end{lem}
\begin{rem}\label{rem: reducetolinear}
One can see from our proof that the result actually holds for $T_{\epsilon}=e^{o(\epsilon^{-1})}$.
\end{rem}

and 
\begin{lem}\label{lem: errorcontrol}
Let $u_{0}$ be as in \eqref{eq: u0}, and $u_{app}$ be as in \eqref{eq: uapp}.  Let $u$ be the associated solution to \eqref{eq: main} with initial data $u_{0}$. Let $\theta>0$ be fixed.  Let $T_{\epsilon}=O(\epsilon^{-1})$. Let $C_{0}\gg 1$. Then there exists $s>\frac{1}{2}$, and $\delta_{1}>0$, so that 
\begin{equation}\label{eq: errorcontrol}
-\ln\mathbb{P}(\|u-u_{app}\|_{L_{t}^{\infty}H^{s}([0,T_{\epsilon}]\times \mathbb{T})}\geq \epsilon^{-\frac{1}{2}+\delta_{1}})\geq  C_{0}\epsilon^{-1}
\end{equation}
 
\end{lem}

Once Lemma \ref{lem: modlinearldp} and Lemma \ref{lem: errorcontrol} are established, Theorem \ref{thm: main} follows by the similar argument  as in \cite{garrido2023large}.
\begin{proof}[Proof of Theorem \ref{thm: main} assuming Lemma \ref{lem: modlinearldp} and Lemma \ref{lem: errorcontrol}.]
We first explain the overall idea.\\
Since we are computing large deviation for $\|u\|_{L^{\infty}}$ at scale $\epsilon^{-1/2}$, any error of form of $\epsilon^{-1/2+\delta}, \delta>0$ is irrelevant. It does take a probability $e^{-C_{0}\epsilon^{-1}}$ to reduce $u$ to $u_{app}$. But here $C_{0}$ can be chosen as arbitrary large.\\
 Note that for any $c$ fixed,  $-\epsilon\ln (e^{-c\epsilon^{-1}}+e^{-C_{0}\epsilon^{-1}})\rightarrow c$ for all $C_{0}$ large than $c$. \\
 
 We present more details for the convenience of readers.
 
 We start with an upper bound for $\mathbb{P}(\left(\sup \limits_{t \in [0,T_{\epsilon}]}\sup \limits_{x \in \mathbb{T}} |u(t,x)| > z_0 \epsilon^{-1/2} \right)$. It is bounded by 
 \begin{equation}\label{eq: twoupper}
 \mathbb{P} \left( \left\| u_{\mathrm{app}} \right\|_{L_{x,t}^\infty } > z_0 \epsilon^{-1/2} - \epsilon^{-1/2 + \delta_1} \right)
 +\mathbb{P}\left( \left\| u - u_{\mathrm{app}}\right\|_{C_t[0,T_{\epsilon}]H^s} > \epsilon^{-1/2 + \delta_1} \right)
 \end{equation} 
 Via Lemma \ref{lem: modlinearldp}, the first term in \eqref{eq: twoupper} is bounded by
 \begin{equation}
 	e^{-\frac{\epsilon ^{-1 } z_0^2}{\sum \limits_{k \in \mathbb{Z}} c_k^2}+o(\epsilon ^{-1})}.
 \end{equation}
 
 Via Lemma \ref{lem: errorcontrol}, the second term is bounded by $e^{-C_{0}\epsilon^{-1}}$ for $C_{0}$ large.
 
 Thus, we have 
 \begin{equation}\label{eq: upperbound}
 	\mathbb{P}(\left(\sup \limits_{t \in [0,T_{\epsilon}]}\sup \limits_{x \in \mathbb{T}} |u(t,x)| > z_0 \epsilon^{-1/2} \right)\leq 2e^{-\frac{\epsilon ^{-1 } z_0^2}{\sum \limits_{k \in \mathbb{Z}} c_k^2}+o(\epsilon^{-1})}
 \end{equation}

 Lower bound is similar, (the point is that we use LDP for the modified linear flow and only need an upper bound of probability to control the error by Lemma \ref{lem: errorcontrol}) and we get 
 \begin{equation}\label{eq: lowerbound}
 	\mathbb{P}\left(\sup \limits_{t \in [0,T_{\epsilon}]}\sup \limits_{x \in \mathbb{T}} |u(t,x)| > z_0 \epsilon^{-1/2} \right)\geq \frac{1}{2}e^{-\frac{\epsilon ^{-1 } z_0^2}{\sum \limits_{k \in \mathbb{Z}} c_k^2}+o(\epsilon^{-1})}
 \end{equation}

 Combine \eqref{eq: lowerbound} and \eqref{eq: upperbound}, the desired LDP estimate for $u$ follows.

\end{proof}

By \cite[Lemma 4.2]{garrido2023large} or \cite[Lemma 4.2]{Tzvetkov2017quasi}, $\eta$ is a standard complex Gaussian if and only $\eta e^{it|\eta|^{2}}$ is a standard complex Gaussian. Our proof for Lemma \ref{lem: modlinearldp} actually yields to the same result\footnote{We will not use this result in the proof, and strictly speaking those two results, though same, are not equivalent, but we feel it is good to record this estimate for linear flow here.} for the linear free flow.
\begin{lem}\label{lem: ldpforlinearflow}
	Let $z_{0}>0$, $T_{\epsilon}=O(\epsilon^{-1})$ and $u_{0}$ be as in \eqref{eq: u0}, one has the large deviation principle
\begin{equation}\label{eq: linearldp}
\lim_{\epsilon\rightarrow 0^{+}}\epsilon \ln \mathbb{P}(\|e^{it\Delta}u_{0}\|_{L_{t,x}^{\infty}([0,T_{\epsilon}]\times \mathbb{T})}>z_{0}\epsilon^{-\frac{1}{2}})=-\frac{z_{0}^{2}}{\sum_{n}c_{n}^{2}}	
\end{equation}
\end{lem}

We will prove Lemma \ref{lem: modlinearldp} and Lemma \ref{lem: ldpforlinearflow} in Section \ref{sec: ldplinear}.

Lemma \ref{lem: errorcontrol} mainly follows from the following two lemmas.

The first one is random data local theory

\begin{lem}\label{lem: localrandom}
Let $u, u_{app}, T_{\epsilon}$ be as in Lemma \ref{lem: errorcontrol}. For all (fixed) $0<\delta_{1}\ll 1$ and $\theta>0$. There exists some $s>\frac{1}{2}$, so that the following holds.

Assume that for some $t_{0}\leq T_{\epsilon}$, one has 
\begin{equation}\label{eq: close}
\|u(t_{0})-u_{app}(t_{0})\|_{H^{s}_{x}}\leq \epsilon^{-\frac{1}{2}},	
\end{equation}
then one has up to (extra) probability $e^{-\epsilon^{-1-\delta_{2}}}$, for some $\delta_{2}>0$,
\begin{equation}\label{eq: xsbmain}
\|e^{2it\epsilon^{2}\frac{1}{2\pi}\|u_{0}\|_{2}^{2}}u(t)-e^{2it_{0}\epsilon^{2}\frac{1}{2\pi}\|u_{0}\|_{2}^{2}}e^{i(t-t_{0})\Delta}u(t_{0})\|_{X^{s,b}[t_{0}, t_{0}+1]}\leq \epsilon^{\frac{1}{2}-\delta_{1}} 
\end{equation}

\end{lem}
\begin{proof}
Lemma \ref{lem: localrandom} lies essentially in \cite{bourgain1994periodic,bourgain1996invariant}, see in particular \cite{compaan2021pointwise}, Proposition 4.6, and Lemma 4.9.  All aforementioned works focus on the regularity issue, i.e. the part $s>\frac{1}{2}$. However, its easy to see the size also matches. Since our initial data is at size $\epsilon^{-1/2}$, the nonlinearity is of cubic form $\epsilon |u|^{2}u$, and \textbf{crucially}, we only need to control error at size $\epsilon^{1/2-\delta_{1}}$ rather than critical scale $\epsilon^{1/2}$. This actually freedom, combines the hypercontractivity estimate, Lemma \ref{lem: hyper}, allows a loss of probability $e^{-\epsilon^{-1-\delta_{2}}}$ rather than $e^{-\epsilon^{-1}}$, the latter would be too large for our application.

\end{proof}

The second one is a normal form type transformation, which is also the key in \cite{garrido2023large}, except that here we needs to combine it with the analysis based on Lemma \ref{lem: localrandom}, via computations of same nature as the $X^{s,b}$ type computations in \cite{bourgain1996invariant}.

\begin{lem}\label{lem: bootstrap}
Let $u,u_{app}, u_{0}, T_{\epsilon}$ be as in Lemma \ref{lem: errorcontrol}. Let $C_{1}\gg C_{0} \gg 1$ and $ C_{2}\gg 1$, $s>\frac{1}{2}$.  Let $0<\delta_{1}\ll 1$. 
Let $F_{\epsilon}=\{\omega|\|u_{0}\|_{2}\leq C_{1}\epsilon^{-1/2}\}$. Then there exists a set $E_{\epsilon}$ with probability measure $e^{-\epsilon^{-1-\delta_{3}}}$, for some $\delta_{3}>0$, such that the following holds for all $\omega\in F_{\epsilon}-E_{\epsilon}$:\\

If there holds bootstrap hypothesis, for all $t\leq T\leq T_{\epsilon}$
 \begin{equation}\label{eq: bh}
 	\|u-u_{app}\|_{C_{t}H^{s}([0,T]\times \mathbb{T})}\leq \epsilon^{-\frac{1}{2}-\delta_{1}},
 \end{equation}
then there holds bootstrap estimate
\begin{equation}\label{eq: be}
	\|u-u_{app}\|_{C_{t}H^{s}([0,T]\times \mathbb{T})}\leq \epsilon^{\frac{1}{2}-C_{2}\delta_{1}}
\end{equation}

 \end{lem}

We will present the proof of Lemma \ref{lem: errorcontrol} assuming Lemma \ref{lem: localrandom} and Lemma \ref{lem: bootstrap} in Section \ref{sec: errorcontrol}.

\section{LDP for the linear flow and modified linear flow}\label{sec: ldplinear}
We prove Lemma \ref{lem: ldpforlinearflow} and Lemma \ref{lem: modlinearldp} in this section. The proof are of same nature. We start with Lemma \ref{lem: ldpforlinearflow}, which is slightly easier.

\subsection{Proof of Lemma \ref{lem: ldpforlinearflow}}
We first start with the following classical pointwise Large Deviations Principle for $e^{it\Delta}u_0(x)$.

Recall $u_{0}=\sum_{n}\frac{g_{n}}{\<n\>^{\frac{1}{2}+\theta}}e^{inx}$.

    \begin{lem}\label{lem: pointwiseldp}
        Fix $t,x$, then 
        \begin{align}\label{eq: pointwiseldp}
            \mathbb{P}(|e^{it\Delta}u_0(x)| > z_0\epsilon ^{-\frac{1}{2}})
		    = e^{-\frac{z_0^2\epsilon ^{-1}}{\sum \limits_{n\in \mathbb{Z}} |c_n|^2}}.
        \end{align}
    \end{lem}
    \begin{proof}
        Fix $t,x$, it is well-known that $e^{it\Delta}u_0(x)$ is still a Guassian with mean $0$ and variance $\sum \limits_{n\in \mathbb{Z}} |c_n|^2$. Therefore, $e^{it\Delta}u_0(x)$ follows a Rayleigh distibution and \eqref{eq: pointwiseldp} holds, see also \cite{garrido2023large}.
    \end{proof}
    Now we apply Lemma \ref{lem: pointwiseldp} to obtain Lemma \ref{lem: ldpforlinearflow}.
    We fix $z_0,\epsilon >0$ throughout.

     Applying \eqref{eq: pointwiseldp}, we have the desired lower bound
    \begin{align*}
        \mathbb{P}(\sup_{t\in [0,T]}\sup_{x\in \mathbb{T}} |e^{it\Delta}u_0(x)|>z_0 \epsilon^{-\frac{1}{2}})
		\geq e^{-\frac{z_0^2 \epsilon ^{-1} }{\sum \limits_{n\in \mathbb{Z}} |c_n|^2}}.
    \end{align*}

    Next we turn to the upper bound. Those types of computations originates from \cite{bourgain1996invariant}, and play an important role in \cite{compaan2021pointwise}, in particular proof of Prop 4.1, see also \cite{fan20212d}.\\

    Let us fix dyadic $N\sim \epsilon^{-\frac{1000}{\theta}}$ large.\\
    
    We note that up to probability $e^{-N^{\theta/100}}$, we have
    \begin{equation}\label{eq: tamegaussian}
    |g_{n}|\leq \<n\>^{\frac{\theta}{100}}, |n|\geq N.	
    \end{equation}
   and 
   \begin{equation}\label{eq: tameguassian2}
   |g_{n}|\leq N, |n|\leq N
   \end{equation}

    We claim
    
    \begin{lem}\label{lem: uniformhigh}
    Take dyadic $N\sim \epsilon^{-\frac{1000}{\theta}}$ large, up to probability $e^{-N^{\theta/200}}\ll e^{-\epsilon^{-2}}$,  one has for all dyadic $M\geq N$
    \begin{equation}\label{eq: uniformhigh} 
    \sup_{t\in [0,T_{\epsilon}]}\|P_{M}e^{it\Delta}u_{0}\|_{L^{\infty}}\lesssim M^{-\theta/100}
    \end{equation}
    and in particular
    \begin{equation}\label{eq: uniformhighwork}
    \sup_{t\in [0,T_{\epsilon}]}\|P_{\geq N}e^{it\Delta}u_{0}\|_{L^{\infty}}\lesssim N^{-\theta/100}.
    \end{equation}

    \end{lem}
    
    Furthermore, we claim
    \begin{lem}\label{lem: lowfrequency}
    Under the same assumptions of Lemma \ref{lem: uniformhigh}, up to probability $e^{-N}\ll e^{-\epsilon^{-2}}$, we have 
    \begin{equation}\label{eq: lowhs}
    \|P_{\leq N}u_{0}\|_{H^{10}}\lesssim  N^{100}.	
    \end{equation}
    and in particular 
    \begin{equation}\label{eq: derivative}
    \|\nabla_{t,x}P_{\leq N}e^{it\Delta}u_{0}\|_{L_{t}^{\infty}L^{\infty}_{x}}\lesssim N^{100}	
    \end{equation}

    \end{lem}
    
    We first prove Lemma \ref{lem: ldpforlinearflow} assuming Lemma \ref{lem: uniformhigh}, Lemma \ref{lem: lowfrequency}. 
    \begin{proof}
    We short $e^{it\Delta}u_{0}$ via $u_{lin}$
    First observe, since $N\sim \epsilon^{-1000/\theta}$. For any fixed points $(x_{1},t_{1}),...,(x_{L},t_{L})$, with $L\lesssim N^{1000}$, one has (since polynomial always dominated by exponential), thanks to Lemma \ref{lem: pointwiseldp},
    \begin{equation}\label{eq: multildpposintwise}
  -\ln \mathbb{P}\left(\sup_{i=1,...,L}|u_{lin}(t_{i},x_{i})|>z_{0}\epsilon^{-1/2}\right)=\frac{z_{0}^{2}\epsilon^{-1}}{\sum_{n}|c_{n}|^{2}}+o(\epsilon^{-1})
    \end{equation}

    Secondly, by Lemma \ref{lem: uniformhigh}, (since $N^{-\theta/100}\sim \epsilon^{10}\ll \epsilon^{-1/2}$),  we have
    \begin{equation}
    -\ln \mathbb{P}\left(\sup_{i=1,...,L}|P_{\leq N}u_{lin}(t_{i},x_{i})|>z_{0}\epsilon^{-1/2}\right)=\frac{z_{0}^{2}\epsilon^{-1}}{\sum_{n}|c_{n}|^{2}}+o(\epsilon^{-1})
    \end{equation}

Now, we pose $(x_{1},t_{1}), ...,(x_{L},t_{L})$ even on $[0, T_{\epsilon}]\times \mathbb{T}$, thus for any $t\in [0, T_{\epsilon}]$, one can find (essentially unique)$|(t_{i},x_{i})-(t,x)|\lesssim N^{-500}$, thus by \eqref{eq: derivative}, we have the desired 
\begin{equation}\label{eq: finalldplinear}
	-\ln \mathbb{P}\left(\sup_{t}\|u_{lin}(t,x)\|_{L_{x}^{\infty}}>z_{0}\epsilon^{-1/2}\right)=\frac{z_{0}^{2}\epsilon^{-1}}{\sum_{n}|c_{n}|^{2}}+o(\epsilon^{-1})
\end{equation}

    \end{proof}

We are left with the proof of Lemma \ref{lem: uniformhigh} and Lemma \ref{lem: lowfrequency}.

We first prove Lemma \ref{lem: lowfrequency}. It is almost direct since \eqref{eq: tameguassian2} implies \eqref{eq: lowhs}, and by Sobolev embedding and the equation of linear Schrodinger, we obtain \eqref{eq: derivative}.

We now prove Lemma \ref{lem: uniformhigh}.

It is enough to prove for all $M\geq N$, one has up to probability $e^{-M^{\theta/150}}$, there holds \eqref{eq: uniformhigh}.

This part is similar to our previous proof of Lemma \ref{lem: ldpforlinearflow} assuming Lemma \ref{lem: uniformhigh}, Lemma \ref{lem: lowfrequency}. \\

Note that for any  $(x_{1},t_{1}),...,(x_{L},t_{L})$, with $L\lesssim M^{1000}$, we have
\begin{equation}
\mathbb{P}(\sup_{l}|P_{M}u_{lin}(x_{l},t_{l})|>M^{-\theta/100})\lesssim e^{-M^{\theta/100}}	
\end{equation}

Furthermore, we have, up to probability $e^{-M}$
\begin{equation}
\|P_{M}u_{0}\|_{H^{10}}\leq M^{100}	
\end{equation}
and thus 
\begin{equation}
\|\nabla_{t,x}e^{it\Delta}P_{M}u_{0}\|_{L_{t,x}^{\infty}}\lesssim M^{100}.
\end{equation}

Pose $(x_{1},t_{1}),...,(x_{L},t_{L})$ evenly on $[0,T_{\epsilon}]\times \mathbb{T}$, and the desired estimates follows as we do for \eqref{eq: finalldplinear}.

\subsection{Proof of Lemma \ref{lem: modlinearldp}}
Lemma \ref{lem: modlinearldp} follows almost same as the argument in previous section.

It suffices to show \eqref{eq: modlinearldp} for $b(t,x) := e^{2it\epsilon^{2}\frac{1}{2\pi}\|u_{0}\|_{2}^{2}}u_{app}=\sum_{k}c_k g_k e^{e^{ikx+it\epsilon^{2}c_k^2 |g_k|^2-it k^2}}$. Fix $t$, noting that $\{g_k e^{it\epsilon ^2c_k^2|g_k|^2}\}_{k\in \mathbb{Z}}$ are still i.i.d complex Guassian normal random variables (see \cite[Lemma 4.2]{garrido2023large}), thus Lemma \ref{lem: pointwiseldp} remains valid. \\

The Sobolev norm control in space is same as the free linear flow. The time derivative control is slightly different. But observe, in frequency $k$, the time derivative for the free flow gives an extra $k^{2}$,  ($\partial_{t}(e^{-itk^{2}})e^{ikx}=-ik^{2}\partial_{t}(e^{-itk^{2}})e^{ikx}$). But for the modified linear flow, the time derivative at frequency $k$ gives an extra $\epsilon^{2}c_{k}^{2}|g_{k}|^{2}$,  noting that we are doing ldp for $g_{k}$ at level of $\epsilon^{-1/2}$, so this part is neglectable  compared to $k^{2}$, and previous arguments works line by line same.

\section{Error control}\label{sec: errorcontrol}

We first prove Lemma \ref{lem: errorcontrol} assuming Lemma \ref{lem: localrandom} and Lemma \ref{lem: bootstrap}.

To start, we first observe Lemma \ref{lem: localrandom} implies the following (one may need to slightly change $\delta_{2}$),

\begin{lem}\label{lem: localrandom2}
Let $u, u_{app}, T_{\epsilon}$ be as in Lemma \ref{lem: errorcontrol}. For all (fixed)  $\theta>0$. There exists some $s>\frac{1}{2}$, so that the following holds.

Assume that for some $t_{0}\leq T_{\epsilon}$, one has 
\begin{equation}\label{eq: close2}
\|u(t_{0})-u_{app}(t_{0})\|_{H^{s}_{x}}\leq \epsilon^{-\frac{1}{2}},	
\end{equation}
then one has up to (extra) probability $e^{-\epsilon^{-1-\delta_{2}}}$, for some $\delta_{2}>0$,
\begin{equation}\label{eq: hsmain}
\|u(t)-u_{app}(t)\|_{L_{t}^{\infty}H^{s}[t_{0}, t_{0}+1]}\lesssim \epsilon^{-\frac{1}{2}}
\end{equation}

\end{lem}
\begin{proof}[Proof of Lemma \ref{lem: localrandom2}]
Recall $X^{s,b}$ embeds into $L_{t}^{\infty}H^{s}$, estimate \eqref{eq: xsbmain} implies 
\begin{equation}\label{eq: hswork}
\|e^{2it\epsilon^{2}\frac{1}{2\pi}\|u_{0}\|_{2}^{2}}u(t)-e^{2it_{0}\epsilon^{2}\frac{1}{2\pi}\|u_{0}\|_{2}^{2}}e^{i(t-t_{0})\Delta}u(t_{0})\|_{L_{t}^{\infty}H^{s}[t_{0}, t_{0}+1]}\leq \epsilon^{\frac{1}{2}-\delta_{1}}
\end{equation}

Now we note that 
\begin{equation}\label{eq: workdeco}
\begin{aligned}
e^{2it\epsilon^{2}\frac{1}{2\pi}\|u_{0}\|_{L^{2}}^{2}}(u-upp)(t)=&e^{2it\epsilon^{2}\frac{1}{2\pi}\|u_{0}\|_{L^{2}}^{2}}u(t)-e^{2it_{0}\epsilon^{2}\frac{1}{2\pi}\|u_{0}\|_{L^{2}}^{2}}u(t_{0})\\
+&e^{2it_{0}\epsilon^{2}\frac{1}{2\pi}\|u_{0}\|_{L^{2}}^{2}}e^{i(t-t_{0})\Delta}u(t_{0})-e^{i(t-t_{0})\Delta}u_{app}(t_{0})\\
+&e^{i(t-t_{0})\Delta}u_{app}(t_{0})-e^{2it\epsilon^{2}\|u_{0}\|_{2}^{2}}u_{app(t)}
\end{aligned}
\end{equation}
Note that phase $e^{2it\epsilon^{2}\frac{1}{2\pi}\|u_{0}\|_{L^{2}}^{2}}$ does not change $H^{s}$.
The first term in RHS of \eqref{eq: workdeco} is controlled via \eqref{eq: hswork}. The second term is controlled by \eqref{eq: close2}.

The third term can be controlled via the explicit formula of $u_{app}$,
\begin{equation}
\begin{aligned}	
\|e^{i(t-t_{0})\Delta}u_{app}(t_{0})-e^{2it\epsilon^{2}\|u_{0}\|_{2}^{2}}u_{app(t)}\|_{H^{s}}^{2}
\lesssim \sum_{n} |c_{n}g_{n}|^{2}|t-t_{0}|^{2}(\epsilon^{2}c_{n}^{2}|g_{n}|^{2})^{2}n^{2s}\\
\lesssim \sum_{n} \<n\>^{-3-6\theta}\<n\>^{2s}|g_{n}|^{6}\epsilon^{4}
\end{aligned}
\end{equation}

The desired estimates follows now from Lemma \ref{lem: hyper}, (or by computing the probability $|g_{n}|\leq <n>^{0+}\epsilon^{-(\frac{1}{2}+)}$)
\end{proof}

We now apply Lemma \ref{lem: localrandom2} and (assuming) Lemma \ref{lem: bootstrap} to close the proof of Lemma \ref{lem: errorcontrol}.

\begin{proof}
It follows standard continuity argument. Note that we are doing an upper bound estimate for$\mathbb{P}(\|u-u_{app}\|_{L_{t}^{\infty}H^{s}([0,T_{\epsilon}]\times \mathbb{T})}\geq \epsilon^{-\frac{1}{2}+\delta_{1}})
\mathbb{P}(\|u-u_{app}\|_{L_{t}^{\infty}H^{s}([0,T_{\epsilon}]\times \mathbb{T})}\geq \epsilon^{-\frac{1}{2}+\delta_{1}})$, and $F_{\epsilon}^{c}$ and $E_{\epsilon}$ in Lemma \ref{lem: bootstrap} admits small enough probability for the estimate in Lemma \ref{lem: errorcontrol}. Everytime, we apply Lemma \ref{lem: localrandom2} to extend the solution for extra time interval of Length 1, which will force us to drop $e^{-\epsilon^{-1-\delta_{2}}}$. We need to drop at most $\sim 1/\epsilon$ times, and such a loss of probability is allowed. We leave the left to interested readers. 

\end{proof}
 
We now turn to the proof  of Lemma \ref{lem: bootstrap}.

Lemma \ref{lem: bootstrap} can be reduced to the following Lemma via a Gronwall argument. 

\begin{lem}\label{lem: normalform}
Under the assumption of Lemma \ref{lem: bootstrap}. Assuming bootstrap hypothesis \eqref{eq: bh} for $[0,T]$. There exists a set $E_{\epsilon}$ with probability measure $e^{-\epsilon ^{-1-\delta_{3}}}$ for some $\delta_{3}>0$, such that for all $\omega\in F_{\epsilon}-E_{\epsilon}$: 
one has for some $C,C_3>0$
\begin{equation}\label{eq: normalformestimate}
\|(u-u_{app})(t)\|_{H^{s}}\leq C\epsilon^{2}\|u_{0}\|_{L_{x}^{2}}^{2}\int_{0}^{t}\|u(\tau)-u_{app}(\tau)\|_{H^{s}}d\tau+C\epsilon^{1/2-C_3\delta_{3}}
\end{equation}
	for all $t\leq T$.
\end{lem}

We first prove Lemma \ref{lem: bootstrap} assuming Lemma \ref{lem: normalform}.
\begin{proof}	
It is standard gronwall argument.\\
    Let $A(t):=\|(u-u_{app})(t)\|_{H^{s}}$ and $B(t) := \int_{0}^{t} A(\tau)d\tau$.  By Lemma \ref{lem: normalform}, for all $t\leq T$ and $\omega\in F_{\epsilon}-E_{\epsilon}$, we have 
    \begin{equation}\label{eq: gronwall}
        A(t)\leq CC_{1}^2\epsilon\int_{0}^{t}A(\tau )d\tau+C\epsilon^{1/2-C_3\delta_{3}}.
    \end{equation}
    Thus, we have
    \begin{equation*}
        B(t)'\leq CC_{1}^2\epsilon B(t)+C\epsilon^{1/2-C_3\delta_{3}},
    \end{equation*}
    which implies
    \begin{equation*}
        \frac{d}{dt}(e^{-CC_{1}^2\epsilon t}B(t))
        \leq C\epsilon^{1/2-C_3\delta_{3}}e^{-CC_{1}^2\epsilon t}.
    \end{equation*}
    Note that $B(0)=0$. For all $t\leq T$ and $\omega\in F_{\epsilon}-E_{\epsilon}$, we have
    \begin{equation}\label{eq: gcontrol}
        B(t)
        \leq C_{1}^{-2}\epsilon^{-1/2-C_3\delta_{3}}(e^{CC_{1}^2\epsilon t}-1).
    \end{equation}
    Inserting \eqref{eq: gcontrol} into \eqref{eq: gronwall}, yields \eqref{eq: be} if we take $C_{2}=2C_3$ and $\delta_{1}=\delta_{3}$.
\end{proof}
\begin{rem}
We can obtain $T_{\epsilon}\leq c\epsilon^{-1} |\ln \epsilon|$ for some universal constant with $0<c\ll 1$ as in \cite{garrido2023large}.	
\end{rem}

We now turn to the proof of Lemma \ref{lem: normalform}, which is the most technical part of the article.

\begin{proof}
    Let $w(t,x) = e^{2it\epsilon^{2}\frac{1}{2\pi}\|u_{0}\|_{L_{x}^{2}}^{2}-it\Delta}u(t,x) = \sum \limits_k w_k(t) e^{ikx}$. From \eqref{eq: main}, we have
    \begin{equation}\label{eq: eqforw}
    \begin{aligned}
        \begin{cases}
            i \partial_t w_k = - \epsilon^2  |w_k|^2 w_k + \epsilon^2  \sum \limits_{R_k(k_1,k_2,k_3)} w_{k_1} \overline{w_{k_2}} w_{k_3} e^{-it\Omega} \\
            w_k(0) = \frac{ g_k}{\left\langle k\right\rangle ^{\frac{1}{2}+\theta }}
        \end{cases},
    \end{aligned}
    \end{equation}
    where $\Omega = k_1^2-k_2^2+k_3^2-k^2$, $R_k(k_1,k_2,k_3) = \{(k_1,k_2,k_3)\in \mathbb{Z}^3:k = k_1 - k_2 + k_3, \Omega \neq 0\}$, see also \cite[(4.5)]{garrido2023large}.
    Let $a(t,x) = e^{2it\epsilon^{2}\frac{1}{2\pi}\|u_{0}\|_{L_{x}^{2}}^{2}-it\Delta}u_{app}(t,x) = \sum \limits_k a_k(t) e^{ikx}$. 
    From \eqref{eq: uappformula}, we have 
    \begin{align}\label{eq: eqfora}
        \begin{cases}
            $$i\partial_{t} a_{k} = - \epsilon^{2}|a_{k}|^{2}a_{k} $$ \\
            $$a_k(0) = \frac{ g_k}{\left\langle k\right\rangle ^{\frac{1}{2}+\theta }}$$
        \end{cases},
    \end{align}     
    see also \cite[(4.6)]{garrido2023large}.
    We subtract \eqref{eq: eqfora} from \eqref{eq: eqforw}, then integrate it from $0$ to $t$ and use the fact that $\|(u-u_{app})(t)\|_{H^{s}} = \|(w-a)(t)\|_{H^{s}}$, the square of the LHS of \eqref{eq: normalformestimate} is bounded by
    \begin{equation}\label{eq: fisrtreduction}
        \begin{aligned} 
            \epsilon ^{4} \sum_{k} \left|\int_{0}^{t}\left(|a_{k}|^{2}a_{k} - |w_{k}|^{2}w_{k}\right)(\tau ) d\tau \right|^2 \left\langle k\right\rangle^{2s} 
            + \epsilon ^{4} \sum_{k} \left| \sum_{R_k(k_1,k_2,k_3)} \int_{0}^{t}\left(w_{k_{1}}\overline{w_{k_{2}}}w_{k_{3}}\right)(\tau ) e^{-i\tau \Omega} d\tau \right|^2 \left\langle k\right\rangle^{2s}.
        \end{aligned}
    \end{equation}
    Next, Minkowski's inequality implies that the first term of \eqref{eq: fisrtreduction} is controlled by
    \begin{align*}
        \lesssim \epsilon ^4 \sum_{k\in\mathbb{Z}}\langle k\rangle^{2s}\left(\int_{0}^{t}\left(|a_{k}|^{2} + |w_{k}|^{2}\right)\left|a_{k} - w_{k}\right| d\tau \right)^2
        \lesssim \epsilon^{4}\|u_{0}\|_{L_{x}^{2}}^{4}\left(\int_{0}^{t}\|u(\tau)-u_{app}(\tau)\|_{H^{s}}d\tau\right)^2,
    \end{align*}
    which accounts for the first term in \eqref{eq: normalformestimate}. In order to get the other term, following \cite{garrido2023large}, we first integrate by parts
    \begin{align}\label{eq: integratebyparts}
        \int_{0}^{t} w_{k_{1}}\overline{w_{k_{2}}}w_{k_{3}} e^{-i\tau\Omega} d\tau = -\frac{1}{i\Omega} w_{k_{1}}\overline{w_{k_{2}}}w_{k_{3}} e^{-i\tau\Omega} \Bigg|_{\tau=0}^{\tau=t} 
        + \frac{1}{i\Omega} \int_{0}^{t} \partial_{\tau}\left(w_{k_{1}}\overline{w_{k_{2}}}w_{k_{3}}\right) e^{-i\tau\Omega} d\tau.
    \end{align}
    Therefore, we reduce \eqref{eq: normalformestimate} to the following two estimates: for all $\omega\in F_{\epsilon}-E_{\epsilon}$, one has
    \begin{align}\label{eq: derivativecontrol}
        \sum_{k\in\mathbb{Z}}\langle k\rangle^{2s} \left|\sum_{R_k(k_1,k_2,k_3)} \frac{1}{\Omega} \int_{0}^{t} \partial_{\tau}(w_{k_{1}}\overline{w_{k_{2}}}w_{k_{3}})(\tau ) e^{-i\tau\Omega} d\tau\right|^2
        \lesssim \epsilon ^{-3-2C_3\delta_1}
    \end{align}        
    and        
        \begin{align}\label{eq: boundary}
            \sum_{k\in\mathbb{Z}}\langle k\rangle^{2s} \left|\sum_{R_k(k_1,k_2,k_3)} \frac{1}{\Omega} \left[\left(w_{k_{1}}\overline{w_{k_{2}}}w_{k_{3}}\right)(t) e^{-it\Omega} - \left(w_{k_{1}}\overline{w_{k_{2}}}w_{k_{3}}\right)(0)\right]\right|^2
            \lesssim \epsilon ^{-3-2C_3\delta_1}.
        \end{align}
    \subsection{Proof of \eqref{eq: derivativecontrol}}
        We deal with \eqref{eq: derivativecontrol} first. Using H\"older's inequality, we can bound LHS of \eqref{eq: derivativecontrol} by 
        \begin{align*}
            t \int_{0}^{t} \sum \limits_{k\in\mathbb{Z}}\langle k\rangle^{2s} \left(\sum \limits_{R_k(k_1,k_2,k_3)} \frac{1}{|\Omega|} |\partial_{\tau}(w_{k_{1}}\overline{w_{k_{2}}}w_{k_{3}})(\tau )| \right)^2 d\tau.
        \end{align*}
        Since $t\lesssim T_\epsilon \lesssim \epsilon ^{-1}$, we can reduce \eqref{eq: derivativecontrol} to: for all $\omega\in F_{\epsilon}- E_{\epsilon}$ (for some $E_{\epsilon}$ with admissible small probability), one has
        \begin{align}\label{eq: dersiredppointk}
            \sup_{0\leq \tau\leq t} \sum \limits_{R_k(k_1,k_2,k_3)} \frac{1}{|\Omega|} |\partial_{\tau}(w_{k_{1}}\overline{w_{k_{2}}}w_{k_{3}})(\tau )|
            \lesssim \<k\>^{-s-1/2-\delta_3}\epsilon ^{-1/2-C_3\delta_1}
        \end{align}        
        for some $\delta_3>0$.   
        By \eqref{eq: eqforw} and the symmetry of $k_1\leftrightarrow k_3$, we can bound LHS of \eqref{eq: dersiredppointk} pointwise in $\tau$ by
        \begin{align*} 
            &\epsilon ^{2} \sum_{R_k(k_1,k_2,k_3)} \frac{1}{|\Omega|} |w_{k_{1}}(\tau ){w_{k_{2}}}(\tau )| \cdot|w_{k_{3}}(\tau )|^3 
            + \epsilon ^{2} \sum_{R_k(k_1,k_2,k_3)} \frac{1}{|\Omega|} |w_{k_{1}}(\tau ){w_{k_{3}}}(\tau )| \cdot|w_{k_{2}}(\tau )|^3 \\
            &+ \epsilon ^{2} \sum_{R_k(k_1,k_2,k_3)} \frac{|w_{k_{1}}(\tau ){w_{k_{2}}}(\tau )|}{|\Omega|} 
            \left|\sum_{R_{k_3}(k_1',k_2',k_3')} w_{k_1'} \overline{w_{k_2'}}  w_{k_3'} e^{-i\tau \Omega'}\right|\\
            &+ \epsilon ^{2} \sum_{R_k(k_1,k_2,k_3)} \frac{|w_{k_{1}}(\tau ){w_{k_{3}}}(\tau )|}{|\Omega|} 
            \left|\sum_{R_{k_2}(k_1',k_2',k_3')} w_{k_1'} \overline{w_{k_2'}}  w_{k_3'} e^{-i\tau \Omega''}\right|\\
            &=:\spadesuit _1 + \spadesuit _2 + \spadesuit _3 + \spadesuit _4,
        \end{align*}
        where $\Omega' = (k_1')^2 - (k_2')^2 +(k_3')^2 -k_3^2$ and $\Omega'' = (k_1')^2 - (k_2')^2 +(k_3')^2 -k_2^2$.
        Therefore, we reduce \eqref{eq: dersiredppointk} to: for all $\omega\in F_{\epsilon} -E_{\epsilon}$, one has
        \begin{equation}\label{eq: intergoal}
            \sup_{0\leq \tau\leq t} (\spadesuit_1 + \spadesuit_2 + \spadesuit_3 + \spadesuit_4)
            \lesssim \<k\>^{-s-1/2-\delta_3}\epsilon ^{-1/2-C_3\delta_1}.
        \end{equation}
        From the bootstrap hypothesis \eqref{eq: bh}, we can write 
        \begin{equation}\label{eq: bhdeco}
            w(t,x) = \sum_{n}c_{n}g_{n}e^{inx}e^{it\epsilon^{2}c_{n}^{2}|g_{n}|^{2}}  + h(t,x) = \sum_{n} w_n(t) e^{inx}
        \end{equation}
        with $\|h\|_{C_t[0,T]H^s} < \epsilon ^{-1/2 - \delta _1}$. Therefore, if we denote
        \begin{align}\label{eq: gepsic}
            G_\epsilon ^c:=\{\omega : |c_{k}g_k| \lesssim {\epsilon ^{-1/2 - \delta _1 }}{\left\langle k\right\rangle ^{-\frac{1}{2}-\theta + \delta_4}},k\in \mathbb{Z}\}
        \end{align}
        for some $0<\delta_4\ll 1$ to be chosen later,
        then we have $\mathbb{P} (G_\epsilon) \leq e^{-\epsilon ^{-1-\delta _1}}$. Further, for any $\omega \in F_\epsilon - G_\epsilon $, we have
        \begin{align}\label{eq: wktaupoint}
            \sup_{0\leq \tau\leq t} |w_{k}(\tau )| \lesssim {\epsilon ^{-1/2 - \delta _1 }}{\left\langle k\right\rangle ^{-s}},k\in \mathbb{Z}
        \end{align}        
        thanks to the bootstrap hypothesis \eqref{eq: bh}.
        Next, we claim that we can reduce \eqref{eq: intergoal} to: \\
        
        For every $n$, there exists a set $E_{\epsilon ,n}$ with probability measure $e^{-\<n\>^{\delta_5}\epsilon ^{-1-2\delta_1}}$ and \fbox{$0<\delta_5<\theta $} to be chosen later such that for all $\omega \in F_\epsilon - G_\epsilon - E_{\epsilon ,n}$, one has
        \begin{equation}\label{eq: crucialreduce}
            \sup_{0\leq \tau\leq t} \left|\sum_{R_{n}(n_1,n_2,n_3)} w_{n_1} \overline{w_{n_2}}  w_{n_3} e^{-i\tau (n_1^2-n_2^2+n_3^2-n^2)}\right| 
            \lesssim \<n\>^{-\theta + \delta_5} \epsilon ^{-3/2-3\delta_1}.
        \end{equation}
        Indeed, if \eqref{eq: crucialreduce} holds, then for all $\omega \in F_\epsilon -G_\epsilon - \bigcup _{n} E_{\epsilon ,n}$, using \eqref{eq: wktaupoint}, the LHS of \eqref{eq: intergoal} is bounded by
        \begin{equation}
            \begin{aligned}
                &\lesssim \sum_{R_k(k_1,k_2,k_3)} \frac{\epsilon ^{-1/2-5\delta _1}}{|\Omega|\<k_1\>^{s}  } \left(\frac{1}{\<k_2\>^{s}\<k_3\>^{3s}} + \frac{1}{\<k_2\>^{s}\<k_3\>^{\theta - \delta_5}} +\frac{1}{\<k_3\>^{s}\<k_2\>^{3s}} + \frac{1}{\<k_3\>^{s}\<k_2\>^{\theta - \delta_5}} \right) \\
                &\lesssim \sum_{R_k(k_1,k_2,k_3)} \frac{\epsilon ^{-1/2-5\delta_1} }{|\Omega|\<k_1\>^{s} \<k_2\>^{s} \<k_3\>^{\theta - \delta_5}} 
                + \frac{\epsilon ^{-1/2-5\delta_1} }{|\Omega|\<k_1\>^{s} \<k_3\>^{s} \<k_2\>^{\theta - \delta_5}}.
            \end{aligned}
        \end{equation}
        Now let $E_\epsilon  = \bigcup _{n} E_{\epsilon ,n} \bigcap  G_\epsilon $. We have 
        \begin{equation*}
            \mathbb{P}(E_\epsilon) \leq \mathbb{P}(G_\epsilon) + \sum_n \mathbb{P}(E_{\epsilon ,n}) \leq e^{-\epsilon ^{-1-\delta_1}} + \sum_n e^{-\<n\>^{\delta_5}\epsilon ^{-1-2\delta_1}} \lesssim e^{-\epsilon ^{-1-\delta_1}}.
        \end{equation*}
        So if \eqref{eq: crucialreduce} holds and we take \fbox{$C_3 \geq 5$}, then we reduce \eqref{eq: intergoal} to showing following two inequalities:
        \begin{equation}\label{eq: key1}
            \sum_{R_k(k_1,k_2,k_3)} \frac{1}{|\Omega|\<k_1\>^{s} \<k_2\>^{s} \<k_3\>^{\theta - \delta_5}} 
            \lesssim \<k\>^{-s-1/2-\delta_3}
        \end{equation}
        and
        \begin{equation}\label{eq: key2}
            \sum_{R_k(k_1,k_2,k_3)} \frac{1}{|\Omega|\<k_1\>^{s} \<k_3\>^{s} \<k_2\>^{\theta - \delta_5}} 
            \lesssim \<k\>^{-s-1/2-\delta_3}
        \end{equation}

        We rewrite LHS of \eqref{eq: key1} as 
        \begin{align}\label{eq: dyadickey1}
            \sum_{N_1,N_2,N_3} \sum_{S_k(k_1,k_2,k_3)} \frac{1}{|\Omega|\left\langle k_1\right\rangle ^s \left\langle k_2\right\rangle ^s \<k_3\>^{\theta - \delta_5} },
        \end{align}
        where $N_1,N_2,N_3$ are dyadic numbers and for fixed $k$, we denote
        \begin{equation*}
            S_k(k_1,k_2,k_3):=\{(k_1,k_2,k_3)\in \mathbb{Z}^3:k=k_1-k_2+k_3,k_2\neq k_1,k_3, |k_i|\sim N_i, i=1,2,3\}.
        \end{equation*}
        (Restrictly speaking, $S_k(k_1,k_2,k_3)$ defined above depends on $N_1,N_2,N_3$, we hide this dependence for simplicity.) We always assume $N_1\leq N_3$ due to the symmetry of $k_1\leftrightarrow k_3$.
        \begin{itemize}
            \item If $N_3\gg N_2$ or $N_3\sim N_2\gg N_1$, we have $N_3\gtrsim \<k\>$ and $|\Omega |\sim N_3 |k_2-k_1|$, thus,
            \begin{align*}
                \sum_{N_1,N_2,N_3} \sum_{S_k(k_1,k_2,k_3)} \frac{1}{|\Omega|\left\langle k_1\right\rangle ^s \left\langle k_2\right\rangle ^s \<k_3\>^{\theta - \delta_5} }
                &\lesssim \sum_{N_1,N_2,N_3} N_3^{ -\theta -1+\delta_5}(N_2N_1)^{-s} \sum_{S_k(k_1,k_2,k_3)} |k_2-k_1|^{-1}\\
                &\lesssim \sum_{N_1,N_2,N_3} N_3^{ -\theta -1+\delta_5}(N_2N_1)^{-s} (N_2\wedge N_1) \ln (N_1\vee N_2)\\
                &\lesssim \<k\>^{ -\theta -1+\delta_5}\\
                &\lesssim \<k\>^{ -s -1/2 - \delta_3},
            \end{align*}
            where the last inequality we take $\delta_3,\delta_5$ such that \fbox{$\delta_3 + \delta_5 <\frac{1}{2} + \theta - s$}.
            \item If $N_2\gg N_3$, we have $N_2\gtrsim \<k\>$ and $|\Omega| \sim N_2^2$ and
            \begin{align*}
                \sum_{N_1,N_2,N_3} \sum_{S_k(k_1,k_2,k_3)} \frac{1}{|\Omega|\left\langle k_1\right\rangle ^s \left\langle k_2\right\rangle ^s \<k_3\>^{\theta - \delta_5} }
                \lesssim \sum_{N_1,N_2,N_3} N_2^{-s-2}N_1^{1-s} N_3^{1+ \delta_5 -\theta }
                \lesssim \<k\>^{-2s} 
                \lesssim \<k\>^{ -s -1/2 - \delta_3},
            \end{align*}
            where the last inequality we choose $\delta_3$ such that \fbox{$\delta_3< s-\frac{1}{2}$}.
            \item If $N_3\sim N_2\sim N_1$, we have $N_3\gtrsim \<k\>$ and
            \begin{align*}
                \sum_{N_1,N_2,N_3} \sum_{S_k(k_1,k_2,k_3)} \frac{1}{|\Omega|\left\langle k_1\right\rangle ^s \left\langle k_2\right\rangle ^s \<k_3\>^{\theta - \delta_5} }
                                &\lesssim \sum_{N_1,N_2,N_3} N_3^{ -\theta -2s+\delta_5} \ln^2 N_2\\
                &\lesssim \<k\>^{ -\theta -2s+\delta_5 }\ln^4 \<k\>\\
                &\lesssim \<k\>^{ -s -1/2 - \delta_3},
            \end{align*}
            where the last inequality we take $\delta_3,\delta_5$ such that \fbox{$\delta_3 + \delta_5 <s-\frac{1}{2} + \theta $}.
        \end{itemize}
        Therefore, \eqref{eq: key1} holds. 

        For \eqref{eq: key2}, LHS of \eqref{eq: key2} is bounded by 
        \begin{align}
            \sum_{N_1,N_2,N_3} \sum_{S_k(k_1,k_2,k_3)} \frac{1}{|\Omega|\left\langle k_1\right\rangle ^s \left\langle k_3\right\rangle ^s} .
        \end{align}        
        We also assume $N_1\leq N_3$ due to the symmetry of $k_1\leftrightarrow k_3$.
        \begin{itemize}
            \item If $N_2\gg N_3$, we have $N_2\gtrsim \<k\>$ and $|\Omega| \sim N_2^2$ and 
            \begin{align*}
                \sum_{N_1,N_2,N_3} \sum_{S_k(k_1,k_2,k_3)} \frac{1}{|\Omega|\left\langle k_1\right\rangle ^s \left\langle k_3\right\rangle ^s}
                \lesssim \sum_{N_1,N_2,N_3} N_2^{ -2}(N_3N_1)^{1-s} 
                \lesssim \<k\>^{ -2s}
                \lesssim \<k\>^{ -s -1/2 - \delta_3},
            \end{align*}
            where the last inequality we choose \fbox{$\delta_3< s - \frac{1}{2}$}.
            \item If $N_3\gg N_2$, we have $N_3\gtrsim \<k\>$ and $|\Omega| \sim N_3 |k_3-k|$ and
            \begin{align*}
                \sum_{N_1,N_2,N_3} \sum_{S_k(k_1,k_2,k_3)} \frac{1}{|\Omega|\left\langle k_1\right\rangle ^s \left\langle k_3\right\rangle ^s}
                &\lesssim \sum_{N_1,N_2,N_3} N_3^{-s-1}N_1^{-s} \sum_{S_k(k_1,k_2,k_3)} |k_3-k|^{-1}\\
                &\lesssim \sum_{N_1,N_2,N_3} N_3^{-s-1}N_1^{-s} (N_1\wedge N_2)\ln N_3\\
                &\lesssim \<k\>^{-2s+} \\
                &\lesssim \<k\>^{ -s -1/2 - \delta_3},
            \end{align*}
            where the last inequality we choose \fbox{$\delta_3< s-\frac{1}{2}$}.
            \item If $N_3\sim N_2\gg N_1$, we have $N_3\gtrsim \<k\>$ and $|\Omega| \sim N_2 |k_2-k_3|$ and
            \begin{align*}
                \sum_{N_1,N_2,N_3} \sum_{S_k(k_1,k_2,k_3)} \frac{1}{|\Omega|\left\langle k_1\right\rangle ^s \left\langle k_3\right\rangle ^s}
                &\lesssim \sum_{N_1,N_2,N_3} N_3^{-s}N_2^{-1}N_1^{-s} \sum_{S_k(k_1,k_2,k_3)} |k_2-k_3|^{-1}\\
                &\lesssim \sum_{N_1,N_2,N_3} N_3^{-s-1}N_1^{1-s} \ln N_3\\
                &\lesssim \<k\>^{-2s+} \\
                &\lesssim \<k\>^{ -s -1/2 - \delta_3},
            \end{align*}
            where the last inequality we choose that \fbox{$\delta_3< s-\frac{1}{2}$}.
            \item If $N_3\sim N_2\sim N_1$, we have $N_3\gtrsim \<k\>$ and
            \begin{align*}
                \sum_{N_1,N_2,N_3} \sum_{S_k(k_1,k_2,k_3)} \frac{1}{|\Omega|\left\langle k_1\right\rangle ^s \left\langle k_3\right\rangle ^s}
                &\lesssim \sum_{N_1,N_2,N_3} N_3^{-s}N_1^{-s} \ln ^2N_3\\
                &\lesssim \<k\>^{-2s+} \\
                &\lesssim \<k\>^{ -s -1/2 - \delta_3},
            \end{align*}
            where the last inequality we choose that \fbox{$\delta_3< s-\frac{1}{2}$}.
        \end{itemize}
        Therefore, \eqref{eq: key2} holds. 

        Hence, we reduce \eqref{eq: intergoal} to \eqref{eq: crucialreduce}.
        
        Note that we have $w_k(\tau ) = c_k g_k e^{i\tau \epsilon ^2c_k^2|g_k|^2} + h_k=:I_k + II_k$ for all $k\in \mathbb{Z}$. We can divide LHS of \eqref{eq: crucialreduce} into $8$ terms and denote by $J_i = I$ if we choose $c_{k_i'} g_{k_i'} e^{i\tau \epsilon ^2c_{k_i'}^2|g_{k_i'}|^2}$ for the decomposition term of $w_{k_i'}$, and denote by $J_i = II$ otherwise, for $i=1,2,3$.

        \subsubsection{Case 1: {$J_1=J_2=J_3=I$}}
        
        We note that the subtle point here is that we need to get estimate in the form of $\sup_{\tau}$, but hypercontractivity estimates only works well for (finite) fixed $\tau$, so the trick we prove for linear LDP matters here.\\

        In this case, we reduce \eqref{eq: crucialreduce} to proving that there exists a set $E_{\epsilon ,n}$ with probability measure $e^{-\<n\>^{\delta_5}\epsilon ^{-1-2\delta_1}}$, for all $\omega \in F_\epsilon - G_\epsilon - E_{\epsilon ,n}$, one has
        \begin{equation}\label{eq: estimateIII}
            \begin{aligned}
                &\sup_{0\leq \tau\leq t} \sum_{N_1,N_2,N_3}\left|\sum_{S_{n}(n_1,n_2,n_3)} \frac{g_{n_1}e^{i\tau \epsilon ^2c_{n_1}^2|g_{n_1}|^2} \overline{g_{n_2}}e^{-i\tau \epsilon ^2c_{n_2}^2|g_{n_2}|^2} g_{n_3}e^{i\tau \epsilon ^2c_{n_3}^2|g_{n_3}|^2}}{\<n_1\>^{\frac{1}{2}+\theta }\<n_2\>^{\frac{1}{2}+\theta }\<n_3\>^{\frac{1}{2}+\theta }} e^{-i\tau (n_1^2-n_2^2+n_3^2-n^2)}\right| \\
                &\lesssim \<n\>^{-\theta + \delta_5} \epsilon ^{-3/2-3\delta_1}.
            \end{aligned}
        \end{equation}
        Denote
        \begin{equation*}
            F_{n,N_1,N_2,N_3,\epsilon }(\tau):=\sum_{S_{n}(n_1,n_2,n_3)} \frac{g_{n_1}e^{i\tau \epsilon ^2c_{n_1}^2|g_{n_1}|^2} \overline{g_{n_2}}e^{-i\tau \epsilon ^2c_{n_2}^2|g_{n_2}|^2} g_{n_3}e^{i\tau \epsilon ^2c_{n_3}^2|g_{n_3}|^2}}{\<n_1\>^{\frac{1}{2}+\theta }\<n_2\>^{\frac{1}{2}+\theta }\<n_3\>^{\frac{1}{2}+\theta }} e^{-i\tau (n_1^2-n_2^2+n_3^2-n^2)}.
        \end{equation*}
        Let $\tilde{N}_1,\tilde{N}_2,\tilde{N}_3$ be the decreasing order of $N_1,N_2,N_3$. Therefore, we have $\tilde{N}_1 \gtrsim \<n\>$. So for all $\omega \in F_\epsilon - G_\epsilon - E_{\epsilon ,n}$, we reduce \eqref{eq: estimateIII} to the following estimate
        \begin{equation}\label{eq: estimatelocIII}
            \sup_{0\leq \tau\leq t} \sum_{\tilde N_1,\tilde N_2,\tilde N_3}\left|F_{n,\tilde N_1,\tilde N_2,\tilde N_3,\epsilon }(\tau)\right| 
            \lesssim \<n\>^{-\theta + \delta_5} \epsilon ^{-3/2-3\delta_1}.
        \end{equation}
        Firstly, fix any $\tau$, noting that $g_{n}e^{i\tau \epsilon ^2c_{n}^2|g_{n}|^2-i\tau n^2}$ are still $i.i.d$ standard complex Guassian \cite{garrido2023large}, we have
        \begin{equation*}
            \mathbb{E}(\left|F_{n,\tilde N_1,\tilde N_2,\tilde N_3,\epsilon }(\tau)\right|^2) 
            = 2\sum_{S_{n}(n_1,n_2,n_3)} \frac{1}{\<n_1\>^{1+2\theta}\<n_2\>^{1+2\theta }\<n_3\>^{1+2\theta}}
            \lesssim \tilde{N}_1^{-1-2\theta} \tilde{N}_2^{-2\theta} \tilde{N}_3^{-2\theta},
        \end{equation*}
        Then applying the hypercontractivity of (multi) Guassians, up to $e^{-\epsilon ^{-1-2\delta_1} \tilde{N}_1^{2\delta_5}}$, we have 
        \begin{equation*}
            \left|F_{n,\tilde N_1,\tilde N_2,\tilde N_3,\epsilon }(\tau)\right|
            \lesssim \tilde{N}_1^{-1/2-\theta+3\delta_5} \tilde{N}_2^{-\theta} \tilde{N}_3^{-\theta} \epsilon ^{-3/2-3\delta_1}.
        \end{equation*}
        Now we decompose $[0,t] = \cup_{i=1}^{t\tilde{N}_1^{10}} I_i$ such that $I_j\cap I_k = \varnothing $ if $|j-k|>1$ and $I_j\cap I_{j+1} = t_{j}$ and the length of each $I_i$ equal to $\tilde{N}_1^{-10}$. Therefore, there exists a set $E_{\epsilon ,n,\tilde N_1,\tilde N_2,\tilde N_3}$ with probability measure $t\tilde{N}_1^{10} e^{-\epsilon ^{-1-2\delta_1} \tilde{N}_1^{2\delta_5}}$ such that for all $\omega \in (E_{\epsilon ,n,\tilde N_1,\tilde N_2,\tilde N_3})^c$, we have
        \begin{equation}\label{eq: pointwiselocIII}
            \left|F_{n,\tilde N_1,\tilde N_2,\tilde N_3,\epsilon }(t_j)\right|
            \lesssim \tilde{N}_1^{-1/2-\theta+3\delta_5} \tilde{N}_2^{-\theta} \tilde{N}_3^{-\theta} \epsilon ^{-3/2-3\delta_1},\quad j=1,\cdots,t\tilde{N}_1^{10}.
        \end{equation}        
        Since for all $\omega \in (G_\epsilon)^c $, we have 
        \begin{equation*}
            \epsilon ^2c_{k}^2|g_k|^2 \lesssim {\epsilon ^{1 - 2\delta_1 }} \ll 1,k\in \mathbb{Z}.
        \end{equation*}
        Therefore, for all $\omega \in F_\epsilon - G_\epsilon $, we have
        \begin{equation}\label{eq: partiallocIII}
            \begin{aligned}
                &\sup_{\tau >0}\left|\partial_\tau F_{n,\tilde N_1,\tilde N_2,\tilde N_3,\epsilon }(\tau)\right|\\
                &\leq \sum_{S_{n}(n_1,n_2,n_3)} \frac{|g_{n_1}g_{n_2}g_{n_3}|~|\epsilon ^2c_{n_1}^2|g_{n_1}|^2 -\epsilon ^2c_{n_2}^2|g_{n_2}|^2 +\epsilon ^2c_{n_3}^2|g_{n_3}|^2-(n_1^2-n_2^2+n_3^2-n^2) |}{\<n_1\>^{\frac{1}{2}+\theta }\<n_2\>^{\frac{1}{2}+\theta }\<n_3\>^{\frac{1}{2}+\theta }}\\
                & \lesssim \sum_{S_{n}(n_1,n_2,n_3)} |n_1^2-n_2^2+n_3^2-n^2| \|u_0\|_2^3\\
                &\lesssim \tilde{N}_1^{4} \epsilon ^{-3/2}.
            \end{aligned}
        \end{equation}
        So from \eqref{eq: pointwiselocIII} and \eqref{eq: partiallocIII}, for all $\omega \in F_\epsilon - G_\epsilon  -E_{\epsilon ,n,\tilde N_1,\tilde N_2,\tilde N_3}$, we have
        \begin{equation}\label{uniformlocIII}
            \sup_{j}\sup_{\tau\in I_j}\left|F_{n,\tilde N_1,\tilde N_2,\tilde N_3,\epsilon }(\tau)\right|
            \lesssim \sup_{j} \left|F_{n,\tilde N_1,\tilde N_2,\tilde N_3,\epsilon }(t_j)\right| + \tilde{N}_1^{4} \epsilon ^{-3/2} |I_j|
            \lesssim \tilde{N}_1^{-1/2-\theta+3\delta_5} \tilde{N}_2^{-\theta} \tilde{N}_3^{-\theta} \epsilon ^{-3/2-3\delta_1}.
        \end{equation}
        If we set $E_{\epsilon ,n}: = \bigcup_{\tilde N_1,\tilde N_2,\tilde N_3} E_{\epsilon ,n,\tilde N_1,\tilde N_2,\tilde N_3}$, then we have 
        \begin{equation*}
            \mathbb{P}(E_{\epsilon ,n}) 
            \leq \sum_{\tilde N_1,\tilde N_2,\tilde N_3} \mathbb{P}(E_{\epsilon ,n,\tilde N_1,\tilde N_2,\tilde N_3})
            \lesssim \sum_{\tilde N_1,\tilde N_2,\tilde N_3} t\tilde{N}_1^{10} e^{-\epsilon ^{-1-2\delta_1} \tilde{N}_1^{2\delta_5}} 
            \lesssim \epsilon ^{-1} e^{-\epsilon ^{-1-2\delta_1} \<n\>^{3\delta_5/2}}
            \lesssim e^{-\epsilon ^{-1-2\delta_1} \<n\>^{\delta_5}}.
        \end{equation*}
        Further, \eqref{uniformlocIII} implies for all $\omega \in F_\epsilon - G_\epsilon - E_{\epsilon ,n}$, one has
        \begin{equation}\label{uniformIII}
            \sup_{0\leq \tau\leq t} \sum_{\tilde N_1,\tilde N_2,\tilde N_3}\left|F_{n,\tilde N_1,\tilde N_2,\tilde N_3,\epsilon }(\tau)\right| 
            \lesssim \sum_{\tilde N_1,\tilde N_2,\tilde N_3} \tilde{N}_1^{-1/2-\theta+3\delta_5} \tilde{N}_2^{-\theta} \tilde{N}_3^{-\theta} \epsilon ^{-3/2-3\delta_1}
            \lesssim \<n\>^{-\theta + \delta_5} \epsilon ^{-3/2-3\delta_1}
        \end{equation}
        if \fbox{$\delta_5<\frac{1}{4}$}, where the last inequality we use the fact that $\tilde{N}_1\gtrsim \<n\>$. Therefore, \eqref{uniformIII} implies \eqref{eq: estimatelocIII}.

        \subsubsection{Case 2: {$J_1=J_2=I,J_3=II$}}
        
        In this case, we we reduce \eqref{eq: crucialreduce} to proving that for all $\omega \in F_\epsilon - G_\epsilon$, one has
        \begin{equation}\label{eq: I-I-II}
            \sup_{0\leq \tau\leq t} \sum_{R_{n}(n_1,n_2,n_3)} \frac{|g_{n_1}g_{n_2}h_{n_3}|}{\<n_1\>^{\frac{1}{2}+\theta }\<n_2\>^{\frac{1}{2}+\theta }}  
            \lesssim \<n\>^{-\theta + \delta_5} \epsilon ^{-3/2-3\delta_1}.
        \end{equation}
        Using Cauchy-Schwarz inequality and \eqref{eq: gepsic}, for all $\omega \in F_\epsilon - G_\epsilon$, we have
            \begin{equation}\label{eq: uniformI-I-II}
                \begin{aligned}
                    &\sup_{0\leq \tau\leq t} \sum_{R_{n}(n_1,n_2,n_3)} \frac{|g_{n_1}g_{n_2}h_{n_3}|}{\<n_1\>^{\frac{1}{2}+\theta }\<n_2\>^{\frac{1}{2}+\theta }}\\
                    &\lesssim \|h\|_{C_t[0,T]H^s} \left(\sum_{n_3} \<n_3\>^{-2s} \left|\sum_{R_{n_3,n}(n_1,n_2)} \frac{\epsilon ^{-1-2\delta_1 } }{\left\langle n_1 \right\rangle ^{1/2+ \theta -\delta_4 } \left\langle n_2 \right\rangle ^{1/2 + \theta -\delta_4}}\right|^2\right)^\frac{1}{2}\\
                    &\lesssim \epsilon ^{-3/2 -3\delta_1 } \left(\sum_{n_3} \<n_3\>^{-2s} \left|\sum_{R_{n_3,n}(n_1,n_2)} \frac{1}{\left\langle n_1 \right\rangle ^{1/2+ \theta -\delta_4 } \left\langle n_2 \right\rangle ^{1/2 + \theta -\delta_4}}\right|^2\right)^\frac{1}{2},
                \end{aligned}
            \end{equation}
            where for fixed $n,n_3$, we denote $R_{n_3,n}(n_1,n_2) := \{(n_1,n_2)\in\mathbb{Z}^2:n=n_1-n_2+n_3,n_2\neq n_1,n_3\}$.
            Next, we aim to reduce \eqref{eq: I-I-II} to following
            \begin{equation}\label{eq: mainreductionI-I-II}
                \sum_{R_{n_3,n}(n_1,n_2)} \frac{1}{\left\langle n_1 \right\rangle ^{1/2+ \theta -\delta_4 } \left\langle n_2 \right\rangle ^{1/2 + \theta -\delta_4}}
                \lesssim \left\langle n-n_3 \right\rangle ^{-2\theta +2\delta_5} .
            \end{equation}
            If \eqref{eq: mainreductionI-I-II} holds, noting that
            \begin{equation*}
                \sum_{|n_3|\ll |n|} \<n_3\>^{-2s} \left\langle n-n_3 \right\rangle ^{-4\theta +4\delta_5}
                \lesssim \<n\>^{-4\theta +4\delta_5} 
            \end{equation*}
            and
            \begin{equation*}
                \sum_{|n_3|\gg |n|} \<n_3\>^{-2s} \left\langle n-n_3 \right\rangle ^{-4\theta +4\delta_5}
                \sim \sum_{|n_3|\gg |n|} \<n_3\>^{-2s-4\theta +4\delta_5} 
                \lesssim \<n\>^{-4\theta +4\delta_5} 
            \end{equation*}
            and
            \begin{equation*}
                \sum_{|n_3|\sim |n|} \<n_3\>^{-2s} \left\langle n-n_3 \right\rangle ^{-4\theta +4\delta_5}
                \sim \<n\>^{-2s} \sum_{|n_3|\sim |n|} \<n_3-n\>^{-4\theta +4\delta_5} 
                \lesssim \<n\>^{1-2s-4\theta +4\delta_5},
            \end{equation*}
            then we have
            \begin{equation*}
                \sum_{n_3} \<n_3\>^{-2s} \left\langle n-n_3 \right\rangle ^{-4\theta +4\delta_5}
                \lesssim \<n\>^{-4\theta +4\delta_5},
            \end{equation*}
            which, along with \eqref{eq: mainreductionI-I-II} and \eqref{eq: uniformI-I-II}, yields \eqref{eq: I-I-II}. So we reduce \eqref{eq: I-I-II} to \eqref{eq: mainreductionI-I-II} if we take \fbox{$\delta_5< \theta $}. Similarly, 
            noting that
            \begin{equation*}
                \sum_{|n_1|\ll |n-n_3|} \<n_1\>^{\delta_4-\frac{1}{2}-\theta } \left\langle n-n_3-n_1 \right\rangle ^{\delta_4-\frac{1}{2}-\theta}
                \lesssim \left\langle n-n_3\right\rangle ^{\delta_4-\frac{1}{2}-\theta} \sum_{|n_1|\ll |n-n_3|} \<n_1\>^{\delta_4-\frac{1}{2}-\theta }
                \lesssim \left\langle n-n_3\right\rangle ^{2\delta_4-2\theta} 
            \end{equation*}
            and
            \begin{equation*}
                \sum_{|n_1|\gg |n-n_3|} \<n_1\>^{\delta_4-\frac{1}{2}-\theta } \left\langle n-n_3-n_1 \right\rangle ^{\delta_4-\frac{1}{2}-\theta}
                \sim \sum_{|n_1|\gg |n-n_3|} \<n_1\>^{2\delta_4-1-2\theta } 
                \lesssim \<n-n_3\>^{-2\theta +2\delta_4} 
            \end{equation*}
            and
            \begin{equation*}
                \sum_{|n_1|\sim |n-n_3|} \<n_1\>^{\delta_4-\frac{1}{2}-\theta } \left\langle n-n_3-n_1 \right\rangle ^{\delta_4-\frac{1}{2}-\theta}
                \sim \left\langle n-n_3\right\rangle ^{\delta_4-\frac{1}{2}-\theta} \sum_{|n_1|\sim |n-n_3|} \left\langle n-n_3-n_1 \right\rangle ^{\delta_4-\frac{1}{2}-\theta}
                \lesssim \<n-n_3\>^{-2\theta +2\delta_4} ,
            \end{equation*}
            then we have \eqref{eq: mainreductionI-I-II}
            if we take \fbox{$\delta_5\leq \delta_4$}.

        \subsubsection{Case 3: {$J_3=J_2=I,J_1=II$}}
        Similar to Case 2.
        
        \subsubsection{Case 4: {$J_1=J_3=I,J_2=II$}}
        
        Similar to Case 2.

        \subsubsection{Case 5: {$J_1=I,J_2=J_3=II$}}
        
        In this case, we we reduce \eqref{eq: crucialreduce} to proving that for all $\omega \in F_\epsilon - G_\epsilon$, one has
        \begin{equation}\label{I-II-II}
            \sup_{0\leq \tau\leq t} \sum_{R_{n}(n_1,n_2,n_3)} \frac{|g_{n_1}h_{n_2}h_{n_3}|}{\<n_1\>^{\frac{1}{2}+\theta }}  
            \lesssim \<n\>^{-\theta + \delta_5} \epsilon ^{-3/2-3\delta_1}.
        \end{equation}
        Using Cauchy-Schwarz inequality and \eqref{eq: gepsic}, for all $\omega \in F_\epsilon - G_\epsilon$, we have
            \begin{equation}\label{uniformI-II-II}
                \begin{aligned}
                    &\sup_{0\leq \tau\leq t} \sum_{R_{n}(n_1,n_2,n_3)} \frac{|g_{n_1}h_{n_2}h_{n_3}|}{\<n_1\>^{\frac{1}{2}+\theta }}\\
                    &\lesssim \sup_{0\leq \tau\leq t} \|h\|_{C_t[0,T]H^s} \left(\sum_{n_3} \<n_3\>^{-2s} \left|\sum_{R_{n_3,n}(n_1,n_2)} \frac{\epsilon ^{-1/2-\delta_1 } |h_{n_2}| }{\left\langle n_1 \right\rangle ^{1/2+ \theta -\delta_4 } }\right|^2\right)^\frac{1}{2}\\
                    &\lesssim \epsilon ^{-1/2 -\delta_1 } \|h\|_{C_t[0,T]H^s}^2 \left(\sum_{R_{n}(n_1,n_2,n_3)} \frac{1}{\left\langle n_1 \right\rangle ^{1+ 2\theta -2\delta_4 } \<n_2\>^{2s} \<n_3\>^{2s}}\right)^\frac{1}{2}\\
                    &\lesssim \epsilon ^{-3/2 -3\delta_1 } \left(\sum_{N_1,N_2,N_3} \sum_{S_n(n_1,n_2,n_3)} \frac{1}{\left\langle n_1 \right\rangle ^{2s} \<n_2\>^{2s} \<n_3\>^{2s}}\right)^\frac{1}{2}
                \end{aligned}
            \end{equation}
            if \fbox{$\delta_4<\frac{1}{2}+\theta -s$}.
            Since
            \begin{align*}
                \sum_{N_1,N_2,N_3} \sum_{S_n(n_1,n_2,n_3)} \frac{1}{\left\langle n_1\right\rangle ^{2s} \left\langle n_2\right\rangle ^{2s} \<n_3\>^{2s} }
                \lesssim \sum_{\tilde N_1,\tilde N_2,\tilde N_3} \tilde N_1^{-2s}(\tilde N_2\tilde N_3)^{1-2s} 
                \lesssim \<n\>^{ -2s},
            \end{align*}
            where $\tilde N_1,\tilde N_2,\tilde N_3$ are the decreasing order of $N_1,N_2,N_3$, the LHS of \eqref{uniformI-II-II} is bounded by $\epsilon ^{-3/2 -3\delta_1 } \<n\>^{ -s}$, which implies \eqref{I-II-II}.

        \subsubsection{Case 6: {$J_2=I,J_1=J_3=II$}}
        
        Similar to Case 5.
        
        \subsubsection{Case 7: {$J_3=I,J_1=J_2=II$}}
        
        Similar to Case 5.

        \subsubsection{Case 8: {$J_1=J_2=J_3=II$}}
        
        In this case, we we reduce \eqref{eq: crucialreduce} to proving that for all $\omega \in F_\epsilon - G_\epsilon$, one has
        \begin{equation}\label{eq: II-II-II}
            \sup_{0\leq \tau\leq t} \sum_{R_{n}(n_1,n_2,n_3)} |h_{n_1}h_{n_2}h_{n_3}|
            \lesssim \<n\>^{-\theta + \delta_5} \epsilon ^{-3/2-3\delta_1}.
        \end{equation}
        Using Cauchy-Schwarz inequality and \eqref{eq: gepsic} as in \eqref{uniformI-II-II}, we have
            \begin{equation}\label{eq: pointwiseII-II-II}
                \begin{aligned}
                    \sup_{0\leq \tau\leq t} \sum_{R_{n}(n_1,n_2,n_3)} |h_{n_1}h_{n_2}h_{n_3}|
                    \lesssim \epsilon ^{-3/2 -3\delta_1 } \left(\sum_{N_1,N_2,N_3} \sum_{S_n(n_1,n_2,n_3)} \frac{1}{\left\langle n_1 \right\rangle ^{2s} \<n_2\>^{2s} \<n_3\>^{2s}}\right)^\frac{1}{2}.
                \end{aligned}
            \end{equation}
        Therefore, similar to Case 5, \eqref{eq: pointwiseII-II-II} is enough to imply \eqref{eq: II-II-II}.

        Therefore, combining the above $8$ cases, we obtain \eqref{eq: crucialreduce} and thus we complete the proof of \eqref{eq: derivativecontrol}.

        Finally, we reduce Lemma \ref{lem: normalform} to \eqref{eq: boundary}.

        \subsection{Proof of \eqref{eq: boundary}}

        It suffices to prove \eqref{eq: boundary} for all $\omega\in F_{\epsilon}-G_{\epsilon}$, there exists some $\delta_6>0$ such that
        \begin{align}\label{eq: uniformbound}
            \sup_{0\leq \tau \leq t} \sum_{R_k(k_1,k_2,k_3)} \frac{|w_{k_{1}}\overline{w_{k_{2}}}w_{k_{3}}(\tau )|}{|\Omega|}
            \lesssim \epsilon ^{-3/2-C_3\delta_1} \<k\>^{-s-1/2-\delta_6}.
        \end{align}
        Using \eqref{eq: wktaupoint}, for all $\omega\in F_{\epsilon}-G_{\epsilon}$, LHS of \eqref{eq: uniformbound} is bounded by 
        \begin{equation}\label{eq: uniformboundreduction}
            \epsilon ^{-3/2-3\delta_1} \sum_{R_k(k_1,k_2,k_3)} \frac{1}{|\Omega|\<k_{1}\>^s \<k_{2}\>^s \<k_{3}\>^s}.
        \end{equation}
        Therefore, since we have taken \fbox{$C_3 \geq 5$}, we reduce \eqref{eq: uniformbound} to bound \eqref{eq: uniformboundreduction} by $\epsilon ^{-3/2-3\delta_1} \<k\>^{-s-1/2-\delta_6}$. This can be implied by \eqref{eq: key1} and \eqref{eq: key2} if we take $\delta_6 = \delta_4$.

        Hence, we complete the proof of Lemma \ref{lem: normalform}.

\end{proof}

\bibliographystyle{plain}
\bibliography{BG}
\end{document}